\documentclass{amsart}
\usepackage[T1]{fontenc}
\usepackage[utf8]{inputenc}
\usepackage{enumerate}
\usepackage{amsthm}
\usepackage{amscd}

\def\car{{\mathbf 1}}

\def\l2{l^2(\N)}

\def\/{\, | \,}
\def\e{\varepsilon}
\def\ee{\e}
\def\C{C}
\def\I{{\mathcal I}}
\newcommand{\HS}{\text{\tiny{HS}}}
\newcommand{\Lip}{\operatorname{Lip}}
\newcommand{\Hom}{\operatorname{Hom}}
\newcommand{\Dom}{\operatorname{Dom}}
\newcommand{\id}{\operatorname{Id}}

\newcommand{\Lin}{\operatorname{Lin}}
\newcommand{\Hyp}{\operatorname{Hyp}}
\def\TT{{\mathfrak T}}

\def\D{{\mathbb D}}
\def\GD{D^\sharp}

\def\R{{\mathbf R}}
\def\L{{\mathbf L}}

\def\Proba0{{\mathcal P}_0}
\def\1e{\frac{1}{\e}}

\def\L{{\mathcal L}}

\def\e{\varepsilon}
\def\/{\, | \,}

\def\ee{\epsilon}

\def\d{\text{ d}}

\def\RE{{\mathbf R}} % L'ensemble des réels
\def\car{{\mathbf 1}}
\def\I{{\mathcal I}}

\def\car{{\mathbf 1}}

\def\P{{\text P}}% Symbole pour probabilite
\def\D{{\mathbb D}}

\def\<{\langle}
\def\>{\rangle}
\def\({\Bigl(}
\def\){\Bigr)}

\def\D{{\mathbb D}}
\def\CC{{\mathcal C}}
\def\DD{{\mathcal D}}
\def\DDp{{\mathcal D}^\prime}

\def\<<{\langle\!\langle\,}
\def\>>{\,\rangle\!\rangle}
\def\embed{{\mathfrak J}_\beta}
\renewcommand{\Dom}{\operatorname{Dom}}
\newcommand{\Id}{\operatorname{Id}}

\newcommand{\Hol}{\operatorname{Hol}}

\def\d{\, \text{d}}
\newcommand{\trace}{\operatorname{trace}}
\newcommand{\N}{{\mathbf N}}
\newcommand{\lemref}[1]{Lemma [\ref{#1}]}
\newcommand{\thmref}[1]{Theorem [\ref{#1}]}

\newtheorem{theorem}{Theorem}

\newtheorem{lemma}{Lemma}
\newtheorem{defn}{Definition}
\newtheorem{hyp}{Hypothesis}

\newtheorem{rem}{Remark}

\def\esp#1#2{\mathbb{E}_{#1}\left[ #2 \right]}
\begin{document}
\title{Stein's method for {B}rownian approximations}
\author{L. Coutin}
\address{Institute  of Mathematics\\
  Universit\'e Toulouse 3\\
  Toulouse, France} \email{laure.coutin@math.univ-toulouse.fr}
\author{L. Decreusefond}
\address{Institut Telecom, Telecom ParisTech, CNRS LTCI\\
  Paris, France} \email{Laurent.Decreusefond@telecom-paristech.fr}
\thanks{This work was motivated by discussions during the Stein's
  program held at Singapore. The second author would like to thank the
  National University of Singapore for warm hospitality and generous
  support. } \thanks{Both authors were partially supported by ANR-10-BLAN-0121.}
\begin{abstract}
Motivated by a theorem of Barbour, we revisit some of the classical
limit theorems in probability from the viewpoint of the Stein
method. We setup the
framework to bound  Wasserstein distances between
some distributions on infinite dimensional spaces. We show that the
convergence rate for the Poisson approximation of the Brownian motion
is as expected proportional to $\lambda^{-1/2}$ where $\lambda$ is the
intensity of the Poisson process. We also exhibit the speed of
convergence for the Donsker Theorem and for the linear interpolation
of the Brownian motion. 
\end{abstract}
\maketitle{}

\section{Introduction}
\label{sec:introduction}
Among the classics in probability theory, one can cite the
approximation in distribution of a Brownian motion by a normalized
compensated Poisson process of intensity going to infinity
or the celebrated Donsker theorem which says that a symmetric random
walk conveniently normalized approaches a Brownian motion in
distribution. Though the topology of the convergence in distribution
is known to derive from a distance on the space of probability
measures, to the best of our knowledge, we are aware of only one
result precising the speed of convergence in one of these two
theorems.  In \cite{MR1035659}, Barbour estimated the distance between
the distribution of a normalized compensated Poisson process of intensity $\lambda$  and the distribution of a Brownian motion.  The
common space on which these two processes are compared is taken as the
space of rcll functions, denoted by ${\mathfrak D}([0,\, 1], \,\R)$
equipped with the distance:
\begin{equation*}
  d_0(\omega, \, \eta)=\inf_{\Phi\in \Hom([0,1])} (\| \omega\circ \Phi -\eta\|_\infty + \|\Phi-\id\|_\infty),
\end{equation*}
where $\Hom([0,1])$ is the set of increasing homeomorphisms of
$[0,1]$.  It is proved in \cite{MR1035659} that the speed of
convergence is not $\lambda^{-1/2}$ as expected but
that there exists a non negligible corrective term. This additional term exists
because the sample-paths of the two processes do not really belong to
the same space: Continuous functions are a rather special class of
rcll functions and sample-paths of Poisson process even normalized are
never continuous whatever the value of the intensity. Thus there seems
to be an
unavoidable gap between the two kind of trajectories in the considered
approximation. Actually, the additional term is related to the modulus
of continuity of the Brownian motion, i.e. in some sense, it measures
the cost to approximate a continuous function by a purely
discontinuous one.

We circumvent this problem by considering Poisson and Brownian
sample-paths as elements of the same space. In fact, the Poisson
sample-paths, like the trajectories of the other processes we are
considering in this paper, belong to a much smaller space than
${\mathfrak D}([0,\, 1], \,\R)$: They all are piecewise
differentiable, i.e. of the form $\sum_{n\in \N}
r_n(t-t_n)\car_{[t_n,\, t_{n+1})}$ where $(t_n,\, n\ge 1)$ is an
increasing sequence of real and $r_n$ are differentiable functions. An
indicator function is not continuous but it has more property than
being rcll. In particular, it belongs to $\I_{\beta,\, p}$ for any
$p\ge 1$ and any $\beta<1/p$ (see below 
for the definition of  $\I_{\beta,\, p}$). On the other hand,
Brownian trajectories are $(1/2-\epsilon)$-Hölder continuous so that
they belong to $\I_{\beta,\, p}$ any $p\ge 1$ and any
$\beta<1/2$. Therefore, the natural candidates to support both the
distribution of piecewise differentiable processes and that of the
Brownian motion are the spaces $\I_{\beta,\, 2}$ for any
$\beta<1/2$. The original problem is then reduced to the computation
of the distance between between a given distribution and a Gaussian
law on some Hilbert space.

The Stein method is known for a long time to give the speed of
convergence of many Gaussian approximations (see for instance
\cite{MR2732624} and references therein). The usual approach requires
some sort of coupling to derive the pertinent estimates. It is only
recently that the mixing of Stein approach and Malliavin calculus
proved its efficiency (see \cite{Nourdin:2012fk} for a thorough
analysis of this line of thought): The search of ad-hoc couplings in
the Stein method is here bypassed by using integration by parts
formula in the sense of Malliavin calculus. In particular, it has been
used for approximations of point processes functionals
\cite{Decreusefond:2008jc,Decreusefond:2012sys,taqqu}. But to the best
of our knowledge, up to the notable exceptions of \cite{MR1035659} and \cite{Shih20111236}, all
these investigations consider finite dimensional Gaussian random
variables. We here develop the framework for a Stein theory on Hilbert
spaces thus circumventing many of the technicalities of
\cite{Shih20111236} which considers Banach valued random
variables. 
Our approach  requires two types of Malliavin \textit{gradients} : One
used to characterize  the target (Gaussian) measure, one built on the
probability space of the measure to be compared to the Gaussian
measure, used to perform the necessary integration by parts. For the
impatient reader, the actual method can be explained informally in dimension $1$.
Imagine that we want to precise the speed of convergence of the
well-known  limit
in distribution:
  \begin{equation*}
      \frac{1}{\sqrt{\lambda}}(X_\lambda-\lambda) \xrightarrow{\lambda \to \infty} {\mathcal N}(0,\, 1),
    \end{equation*}
where $X_\lambda$ is a Poisson random variable of parameter $\lambda$.
We consider the Wasserstein distance between the distribution of
$\tilde{X}_\lambda=\lambda^{-1/2}(X_\lambda-\lambda)$ and ${\mathcal
  N}(0,\, 1)$, which is defined as
\begin{equation}\label{eq_gaussapp20:1}
  \text{dist}_W(\tilde{X}_\lambda,\, {\mathcal
  N}(0,\, 1))=\sup_{F\in 1-\Lip} \esp{}{F(\tilde{X}_\lambda)}-\esp{}{F({\mathcal N}(0,\, 1))},
\end{equation}
where $1-\Lip$ is the set of one Lipschitz function from $\R$ into itself.
The well known Stein Lemma stands that for any $F\in 1-\Lip$, there exists
$H_F\in {\mathcal C}^2_b$ such that for all $x\in \R$,
 \begin{equation*}
      F(x)-\esp{}{F({\mathcal N}(0,\, 1))}=x\,
      H_F(x)-H_F^\prime(x). 
    \end{equation*}
Moreover, 
 \begin{equation*}
      \Vert H_F^\prime\|_\infty \le 1, \   \Vert H_F^{\prime\prime}\|_\infty \le 2
    \end{equation*}
Hence, instead of the right-hand-side of \eqref{eq_gaussapp20:1}, we
are lead to estimate 
\begin{equation}\label{eq_gaussapp20:3}
  \sup_{\Vert H^\prime\|_\infty \le 1, \   \Vert
    H^{\prime\prime}\|_\infty \le 2}\esp{}{\tilde{X}_\lambda H(\tilde{X}_\lambda)-H^\prime(\tilde{X}_\lambda)}.
\end{equation}
This is where the Malliavin-Stein approach differs from the classical
line of thought. In order to transform the last expression, instead of
constructing a coupling, we  resort to the
integration by parts formula for functionals of Poisson random
variable. The next formula can be checked by hand or viewed as a
consequence of \eqref{eq_gaussapp6:4}:
 \begin{equation*}
      \esp{}{\tilde{X}_\lambda G(\tilde{X}_\lambda)}=\sqrt{\lambda} \ \esp{}{G(\tilde{X}_\lambda+1/\sqrt{\lambda})-G(\tilde{X}_\lambda)}.
    \end{equation*}
Hence, \eqref{eq_gaussapp20:3} is transformed into
\begin{equation}\label{eq_gaussapp20:2}
  \sup_{\Vert H^\prime\|_\infty \le 1, \   \Vert
    H^{\prime\prime}\|_\infty \le 2} \esp{}{\sqrt{\lambda}(H(\tilde{X}_\lambda+1/\sqrt{\lambda})-H(\tilde{X}_\lambda))-H^\prime(\tilde{X}_\lambda)}.
\end{equation}
According to the Taylor formula
    \begin{equation*}
      H(\tilde{X}_\lambda+1/\sqrt{\lambda})-H(\tilde{X}_\lambda)=\frac{1}{\sqrt{\lambda}}H^\prime(\tilde{X}_\lambda)+\frac{1}{2\lambda}H^{\prime\prime}(\tilde{X}_\lambda+\theta/\sqrt{\lambda}),
    \end{equation*}
where $\theta\in (0,1)$. If we plug this expansion into
\eqref{eq_gaussapp20:2}, the term containing $H^\prime$ is
miraculously vanishing and we are left with only the second order term. This
leads to the estimate (compare to Theorem~\ref{thm_gaussapp10:2}):
 \begin{equation*}
         \text{dist}_{W}\left( \tilde{X}_\lambda,\
     {\mathcal N}(0,\, 1)\right) \le \frac{1}{\sqrt{\lambda}}\cdotp
      \end{equation*}
The remainder of this paper consists in generalizing these
computations to the infinite dimensional setting.
 We show that our method is applicable in three different settings: Whenever
the alea on which the approximate process is built upon is either the
Poisson space, the Rademacher space or the Wiener space.

This paper is organized as follows. After some preliminaries, we  construct the Wiener
measure  on the Besov-Liouville spaces and $\l2$, using the It\^o-Nisio Theorem. Section \ref{sec:stein-method} is
devoted to the development of the abstract version of the Stein method
for Hilbert valued random variables. In
Section~\ref{sec:norm-appr-poiss} to Section
\ref{sec:donsker-theorem}, we exemplify this general scheme of
reasoning successively for the Poisson approximation of the Brownian
motion, for the linear interpolation of the Brownian motion and for
the Donsker theorem.  In Section \ref{sec:transfer-principle}, we show that by a
transfer principle, similar results can be obtained for other Gaussian
processes like the fractional Brownian motion, extending some earlier
results \cite{decreusefond03.2}. 

\section{Preliminaries}
\label{sec:preliminaries}

\subsection{Tensor products of Hilbert spaces}
\label{sec:tens-prod-hilb}
For $X$ and $Y$ two Hilbert spaces, ${\mathfrak B}(X,Y)$ is the set of
multilinear complex-valued forms over $X\times Y$. For $x\in X$ and 
$y\in Y$, the bilinear form $x\otimes y$ is defined by
\begin{equation*}
  x\otimes y(f,g)=\langle x,\, f\rangle_X\, \langle y,\, g\rangle_Y,
\end{equation*}
for any $f\in X$ and $g\in Y$. We denote by ${\mathfrak B}_f(X,Y)$,
the linear span of such simple bilinear forms. It is equipped with the
norm
\begin{equation*}
  \| \sum_{i=1}^N \alpha_i x_i\otimes y_i\|_{{\mathfrak B}_f(X,Y)}^2=\sum_{i=1}^N |\alpha_i|^2 \|x_i\|_X^2\|y_i\|_Y^2.
\end{equation*}
The tensor product $X\otimes Y$ is the completion of  ${\mathfrak
  B}_f(X,Y)$ with respect to this norm. A continuous linear map $A$ from $X$ to $Y$ can be viewed as an element of $X\otimes Y$ by the identification :
\begin{equation*}
  \tilde A (f,g)=\langle Af,\, g\rangle_Y \text{ for } f\in X \text{ and } g\in Y.
\end{equation*}
Conversely,  for $x\in X$ and $ y\in Y$, the operator $x\otimes y$ can be seen
either as an element of $X\otimes Y$ or as a continuous map from $X$
into $Y$ via the identification : $$x\otimes y(f)=\langle x,\, f\rangle_X\,
y.$$
We recall that for $X$ an Hilbert space and $A$ a linear continuous map from $X$ into itself, $A$ is said to be trace-class whenever
the series $\|A\|_{\mathcal S_1}:=\sum_{n\ge 1} |(Af_n,\, f_n)_X|$ is convergent for one (hence 
any) complete orthonormal basis $(f_n,\,n\ge 1)$ of $X$. When $A$ is
trace-class, its trace is defined as $\trace(A)=\sum_{n\ge 1} (Af_n,\,
f_n)_X$. It is then straightforward that for $x,\, y\in X$,  the operator $x\otimes y$ is trace-class
and that $\trace(x\otimes y)=\sum_{n\ge 1}(y,\, f_n)_X(x,\,
f_n)_X=\langle x,\, y\rangle_X$ according to the Parseval formula. The trace-class
operators is a two sided ideal of the set of bounded compact
operators: If $A$ is trace-class and $B$ is bounded, then $A\circ B$
is trace-class and (see \cite{MR2154153})
\begin{equation}\label{eq_gaussapp_first:4}
|\trace(A\circ B)| \le \|A\circ B\|_{\mathcal S_1}\le \|A\|_{\mathcal S_1}\|B\|,
\end{equation}
where $\|B\|$ is the operator norm of $B$.
It is easily seen that when $A$ is symmetric and non-negative, $\|A\|_{\mathcal S_1}$ is
equal to $\trace(A)$.
We also need to introduce the notion of partial trace. For any
vector space $X$, $\Lin(X)$ is the set of linear operator from into
itself. For $X$ and $Y$ two Hilbert spaces, the partial trace operator
along $X$ can be defined as follows: it is the unique linear operator
\begin{equation*}
  \trace_X \, :\, \Lin(X\otimes Y) \longrightarrow \Lin(Y)
\end{equation*}
such that for any $R\in \Lin(Y)$, for any trace class operator $S$ on
$X$,
\begin{equation*}
  \trace_X(S\otimes R)=\trace_X(S)\, R.
\end{equation*}

  \subsection{Besov-Liouville spaces}
  \label{sec:slobodetzki-spaces}

  This part is devoted to the presentation of the so-called
  Besov-Liouville spaces.
  A complete exposition can be found in \cite{samko93}.
  For $f\in \L^p([0,1];\ dt),$ (denoted by $\L^p$ for short) the left
  and right fractional integrals of $f$ are defined by~:
  \begin{align*}
    (I_{0^+}^{\alpha}f)(x) & =
    \frac{1}{\Gamma(\alpha)}\int_0^xf(t)(x-t)^{\alpha-1}\d t\ ,\ x\ge
    0,\\
    (I_{1^-}^{\alpha}f)(x) & =
    \frac{1}{\Gamma(\alpha)}\int_x^1f(t)(t-x)^{\alpha-1}\d t\ ,\ x\le 1,
  \end{align*}
  where $\alpha>0$ and $I^0_{0^+}=I^0_{1^-}=\Id.$ For any $\alpha\ge
  0$, $p,q\ge 1,$ any $f\in \L^p$ and $g\in \L^q$ where
  $p^{-1}+q^{-1}\le \alpha +1$, we have~:
  \begin{equation}
    \label{int_parties_frac}
    \int_0^1 f(s)(I_{0^+}^\alpha g)(s)\d s = \int_0^1 (I_{1^-}^\alpha 
    f)(s)g(s)\d s.
  \end{equation}
  For $p \in [1,+\infty],$ the Besov-Liouville space
  $I^\alpha_{0^+}(\L^p):= \I_{\alpha,p}^+$ is usually equipped with
  the norm~:
  \begin{equation}
    \label{normedansIap}
    \|  I^{\alpha}_{0^+}f \| _{ \I_{\alpha,p}^+}=\| f\|_{\L^p}.
  \end{equation}
  Analogously, the Besov-Liouville space $I^\alpha_{1^-}(\L^p):=
  \I_{\alpha,p}^-$ is usually equipped with the norm~:
  \begin{equation*}
    \| I^{-\alpha}_{1^-}f \| _{ \I_{\alpha,p}^-}=\|  f\|_{\L^p}.
  \end{equation*}
  We then have the following continuity results (see
  \cite{feyel96,samko93})~:
  \begin{theorem}
    \label{prop:proprietes_int_RL}
    \begin{enumerate}[i.]
    \item \label{inclusionLpLq} If $0<\alpha <1,$ $1< p <1/\alpha,$
      then $I^\alpha_{0^+}$ is a bounded operator from $\L^p$ into
      $\L^q$ with $q=p(1-\alpha p)^{-1}.$
    \item For any $0< \alpha <1$ and any $p\ge 1,$ $\I_{\alpha,p}^+$
      is continuously embedded in $\Hol_0(\alpha- 1/p)$ provided that
      $\alpha-1/p>0.$ $\Hol_0(\nu)$ denotes the space of $\alpha$
      H\"older-continuous functions, null at time $0,$ equipped with
      the usual norm.
    \item For any $0< \alpha< \beta <1,$ $\Hol_0 (\beta)$ is compactly
      embedded in $\I_{\alpha,\infty}.$
    \item \label{semigroupe} By $I^{-\alpha}_{0^+},$ respectively
      $I^{-\alpha}_{1^-},$ we mean the inverse map of
      $I^{\alpha}_{0^+},$ respectively $I^{\alpha}_{1^-}.$ The
      relation
      $I^{\alpha}_{0^+}I^{\beta}_{0^+}f=I^{\alpha+\beta}_{0^+}f,$
      respectively
      $I^{\alpha}_{1^-}I^{\beta}_{1^-}f=I^{\alpha+\beta}_{1^-}f,$
      holds whenever $\beta >0,\ \alpha+\beta >0$ and $f\in \L^1.$
    \item For $\alpha p<1,$ the spaces $\I_{\alpha,p}^+$ and
      $\I_{\alpha,p}^{-}$ are canonically isomorphic. We will thus use
      the notation $\I_{\alpha,p}$ to denote any of this spaces.
    \end{enumerate}
  \end{theorem}
  We now recall the definition and properties of Besov-Liouville
  spaces of negative orders. The proofs can be found in
  \cite{decreusefond03.1}.

  Denote by $\DD_+$ the space of ${\mathcal C}^\infty$ functions
  defined on $[0,1]$ and such that $\phi^{(k)}(0)=0,$ for all $k \in
  \N.$ Analogously, set $\DD_-$ the space of ${\mathcal C}^\infty$
  functions defined on $[0,1]$ and such that
  $\phi^{(k)}(1)=0$, for all $k \in \N.$ They are both equipped with
  the projective topology induced by the semi-norms
  $p_k(\phi)=\sum_{j\le k}\lVert \phi^{(j)} \rVert_\infty,~~\forall k
  \in \N.$ Let $\DD^\prime_+,$ resp. $\DD^\prime_-,$ be their strong
  topological dual. It is straightforward that $\DD_+$ is stable by
  $I^\beta_{0^+}$ and $\DD_-$ is stable $I^\beta_{1^-},$ for any
  $\beta\in \RE^+.$ Hence, guided by (\ref{int_parties_frac}), we can
  define the fractional integral of any distribution (i.e., an element
  of $\DDp_-$ or $\DDp_+$):
  \begin{align*}
    \text{ For }T\in \DDp_-;\ I^\beta_{0^+}T: \ \phi \in\DD_- &
    \mapsto <T,\,
    I^\beta_{1^-}\phi>_{\DDp_-,\DD_-},\\
    \text{ For }T\in \DDp_+;\ I^\beta_{1^-}T: \ \phi \in\DD_+ &
    \mapsto <T,\, I^\beta_{0^+}\phi>_{\DDp_+,\DD_+}.
  \end{align*}
  We introduce now our Besov-Liouville spaces of negative order as
  follows.
  \begin{defn}
    For $\beta > 0$ and $r>1,$ $\I_{-\beta,r}^+$
    (resp. $\I_{-\beta,r}^-$) is the space of distributions
    $T\in\DD_-^\prime$ (resp. $T \in\DD_+^\prime)$ such that $I^\beta_{0^+}T$
    (resp. $I^\beta_{1^-}T$ ) belongs to $\L^r.$ The norm of an
    element $T$ in this space is the norm of $I^\beta_{0^+}T$ in
    $\L^r$ (resp. of $I^\beta_{1^-}T$).
  \end{defn}
  \begin{theorem}
    \label{thm:cardual}
    For $\beta > 0$ and $r>1,$ the dual space of $\I_{\beta,r}^+$
    (resp. $\I_{\beta,r}^-$) is canonically isometrically isomorphic
    to $I^{-\beta}_{1^-}(\L^{r^*})$
    (resp. $I^{-\beta}_{0^+}(\L^{r^*})$,) where $r^*=r(r-1)^{-1}.$

    Moreover, for $\beta\ge \alpha \ge 0$ and $r>1,$ $I_{1^-}^\beta$
    is continuous from $\I_{-\alpha,r}^-$ into
    $\I_{\beta-\alpha,r}^-.$
  \end{theorem}

  The first part of the next theorem is a deep result which can be
  found in \cite{MR0388013}. We complement it by the computation of
  the Hilbert-Schmidt norm of the canonical embedding $\kappa_\alpha$
  from $\I^+_{\alpha,2}$ into $\L^2$.
  \begin{theorem}
    \label{thm:hilbertschmidt}
    The canonical embedding $\kappa_\alpha$ from $\I^-_{\alpha,2}$
    into $\L^2$ is Hilbert-Schmidt if and only if $\alpha >1/2$.
    Moreover,
    \begin{equation}\label{eq_gaussapp20:5}
c_\alpha:=      \|\kappa_\alpha\|_{HS}=\| I^{\alpha}_{0^+}\|_{HS}=\|I^\alpha_{1^-}\|_{HS} = \frac{1}{2\Gamma(\alpha)}
      \left(\frac{1}{\alpha( \alpha-1/2)}\right)^{1/2}.
    \end{equation}
  \end{theorem}
  \begin{proof}
    Let $(e_n,\, n\ge 1)$ be a CONB of $\L^2$ then
    $(h_n^\alpha=I_{1^-}^{\alpha}(e_n),\ n \in \N)$ is a CONB of ${\mathcal
      I}^-_{\alpha, 2}$ and
    \begin{multline*}
      \|\kappa_\alpha\|_{\HS}^2= \sum_{n\ge 1} \|h_n^\beta\|_{\L^2}^2 = \sum_n \|
      I_{1^-}^{\alpha}(e_n)\|_{\L^2}^2
      =\|I_{1^-}^{\alpha}\|_\HS^2\\
      =\frac{1}{\Gamma(\alpha)^2}\iint_{[0,\, 1]^2}
      (t-s)^{2\alpha-2}\d s\d t,
    \end{multline*}
    and the result follows by straightforward quadrature. The same
    reasoning shows also that $c_\alpha=\| I^{\alpha}_{0^+}\|_{HS}$.
  \end{proof}
For any $\tau\in [0,\,1]$, let $\epsilon_\tau$ the Dirac measure at
point $\tau$. In view of \thmref{prop:proprietes_int_RL}, assertion \ref{inclusionLpLq}, $\epsilon_\tau$ belongs to
$(\I_{\alpha,\, 2}^-)^\prime$ for any $\alpha>1/2$. As will be apparent
below, we need to estimate the norm $\epsilon_\tau$ in this space. 
  \begin{lemma}
    \label{lem_gaussapp_first:1}
For any $\alpha>1/2$, for any $\tau\in [0,\,1]$, the image of $\epsilon_\tau$ by $j_\alpha$,
the
canonical isometry between $(\I_{\alpha,\, 2}^-)^\prime$ and $\I_{\alpha,\, 2}^-$, is the
function 
\begin{equation*}
 j_\alpha(\epsilon_\tau)\, :\, s\longmapsto  I^{\alpha}_{1^-}\Bigl((.-\tau)_+^{\alpha-1}\Bigr)(s)
\end{equation*}
and 
\begin{equation}\label{eq_gaussapp_first:1}
  \|j_\alpha(\epsilon_\tau)\|_{\I_{\alpha,\, 2}^-}^2=\sum_{k\in \N}|h^\alpha_k(\tau)|^2=\frac{(1-\tau)^{2\alpha-1}}{(2\alpha-1)\Gamma(\alpha)^2}\cdotp
\end{equation}
  \end{lemma}
  \begin{proof}
    By definition of the dual product, for any $h=I^\alpha_{1^-}(\dot
    h)$ where $\dot h \in \L^2$,
    \begin{multline*}
      \langle \epsilon_\tau,\, h\rangle_{(\I_{\alpha,\, 2}^-)^\prime,\, \I_{\alpha,\,
          2}^-}=h(\tau)=\frac{1}{\Gamma(\alpha)}\int_\tau^1
      (s-\tau)^{\alpha-1}\dot h(s)\d
      s\\ 
=\Bigr(\frac{1}{\Gamma(\alpha)}(.-\tau)_+^{\alpha-1},\ \dot
h\Bigr)_{\L^2}=\langle
I^{\alpha}_{1^-}\Bigl((.-\tau)_+^{\alpha-1}\Bigr),\ h\rangle_{\I_{\alpha,\,
          2}^-,\I_{\alpha,\,
          2}^-},
    \end{multline*}
hence the first assertion. Moreover, according to Parseval identity in
$\L^2$, we have
\begin{multline*}
  \sum_{k\in \N}|h^\alpha_k(\tau)|^2= \sum_{k\in \N}  \langle \epsilon_\tau,\, h_k^\alpha\rangle_{(\I_{\alpha,\, 2}^-)^\prime,\, \I_{\alpha,\,
          2}^-}^2
=\sum_{k\in \N} \langle j_\alpha(\epsilon_\tau),\, h_k^\alpha\rangle_{\I_{\alpha,\, 2}^-,\, \I_{\alpha,\,
          2}^-}^2\\ =\frac{1}{\Gamma(\alpha)^2}\sum_{k\in \N} \Bigr(
      (.-\tau)_+^{\alpha-1},\
      e_k\Bigr)_{\L^2}^2=\frac{1}{\Gamma(\alpha)^2}
      \|(.-\tau)_+^{\alpha-1}\|_{\L^2}^2=\|j_\alpha(\epsilon_\tau)\|_{\I^-_{\alpha,\,
          2}}^2.
\end{multline*}
Then, \eqref{eq_gaussapp_first:1} follows by quadrature.
  \end{proof}
\section{Gaussian structures on Hilbert spaces}
In order to compare quantitatively the distribution of a piecewise
differentiable process with that of a Brownian motion, we need to
consider a functional space to which the sample-paths of both
processes belong to.  Ordinary Brownian motion is known to have
sample-paths Hölder continuous of any order smaller than $1/2$. Thus,
\thmref{prop:proprietes_int_RL} ensures that its sample-paths belongs
to $\I_{\beta,\infty}\subset \I_{\beta,2}$ for any $\beta<1/2$.
Moreover, a simple calculation shows that for any $\alpha\in (0,1)$,
$$\car_{[a,+\infty)}=\Gamma(\alpha)\,I^\alpha_{0^+}((.-a)_+^{-\alpha}).$$  Hence $\car_{[a,+\infty)}$
belongs to $\I_{1/2-\ee,\, 2}$ for any $\ee>0$. This implies that
random step functions belong to $\I_{\beta,2}$ for any $\beta<1/2$.
The space of choice may thus be any space $\I_{\beta,\, 2}$ for any
$\beta<1/2.$ The closer to $1/2$ $\beta$ is, the most significant the
distance is but the the greater the error bound is.

\subsection{Gaussian structure on Besov-Liouville spaces}
\label{sec:gauss-appr-slob}

To construct the Wiener measure on $\I_{\beta,\, 2}$, we start from
the It\^o-Nisio theorem. Let $(X_n,\, n\ge 1)$ be a sequence of
independent centered Gaussian random variables of unit variance defined
on a common probability space $(\Omega,\, {\mathcal A},\, \P)$. Let
$(e_n,\, n\ge 1)$ be a complete orthonormal basis of $\L^2([0,\,
1])$. Then,
\begin{equation*}
  B(t):= \sum_{n\ge 1}X_nI_{0^+}^1(e_n)(t)
\end{equation*}
converges almost-surely for any $t\in [0,\, 1]$. From
\cite{MR0235593}, we already know that the convergence holds uniformly
with respect to $t$ and thus that $B$ is continuous.  To prove that
the convergence holds in $L^2(\Omega;\, \I_{\beta,\, 2})$, it suffices
to show that
\begin{equation}\label{eq_gaussapp10:12}
  \sum_{n\ge 1} \|I_{0^+}^1e_n\|^2_{\I_{\beta,\, 2}}=\sum_{n\ge 1}
  \|I_{0^+}^{1-\beta}e_n\|^2_{\L^2}
  =\|I_{0^+}^{1-\beta}\|_{\HS}^2<\infty. 
\end{equation}
From \thmref{thm:hilbertschmidt}, we know that $ I^{1-\beta}$ is an
Hilbert-Schmidt operator from $\L^2$ into itself if and only if
$1-\beta>1/2$, i.e. $\beta<1/2$. Thus, for $\beta<1/2$, the
distribution of $B$ defines the Wiener measure on $\I_{\beta,\,
  2}$. We denote this measure by $\mu_\beta$.  Note that
\eqref{eq_gaussapp10:12} implies that the embedding from
$\I_{1-\beta,\, 2}$ into $\L^2$ is also Hilbert-Schmidt and that its
Hilbert-Schmidt norm is $\|I_{0^+}^{1-\beta}\|_{\HS}$. 
By the very definition
of the scalar product on $\I_{\beta,\, 2}$, for $\eta\in \I_{\beta,\,
  2}$, we have
\begin{align*}
  \esp{\mu_\beta}{\exp(i\langle \eta,\, \omega\rangle_{\I_{\beta,\, 2}})}&=\esp{\P}{\exp(i\sum_{n\ge 1} \int_0^1 (I_{1^-}^{1-\beta}\circ I_{0^+}^{-\beta})\eta(s)\ e_n(s)\d s \ X_n)}\\
  &=\exp(-\frac 12\sum_{n\ge 1} \left(\int_0^1 (I_{1^-}^{1-\beta}\circ I_{0^+}^{-\beta})\eta(s)\ e_n(s)\d s\right)^2)\\
  &=\exp(-\frac 12 \|(I_{1^-}^{1-\beta}\circ I_{0^+}^{-\beta})\eta\|_{\L^2}^2)\\
  &=\exp(-\frac 12 \int_0^1 (I_{0^+}^{1-\beta}\circ I_{1^-}^{1-\beta})
  \dot\eta(s)\ \dot\eta(s)\d s),
\end{align*}
where $\dot\eta$ is the unique element of $\L^2$ such that
$\eta=I_{0^+}^\beta \dot\eta.$ Thus, $\mu_\beta$ is a Gaussian measure
on $\I_{\beta,\, 2}$ of covariance operator given by
\begin{equation*}
  V_\beta=I_{0^+}^\beta\circ I_{0^+}^{1-\beta}\circ
  I_{1^-}^{1-\beta}\circ I_{0^+}^{-\beta}.
\end{equation*}
This means that
\begin{equation*}
  \esp{\mu_\beta}{\exp(i\langle \eta,\, \omega\rangle_{\I_{\beta,\,
        2}})}=\exp(-\frac 12 \langle V_\beta\eta,\,
  \eta\rangle_{\I_{\beta,\, 2}}).
\end{equation*}
We could thus in principle make all the computations in $\I_{\beta,\,
  2}$. It turns out that we were not able to be explicit in the
computations of some traces of some involved operators the expressions
of which turned to be rather straightforward in $l^2(\N)$ (where $\N$ is the set of positive integers). This is why
we transfer all the structure to $l^2(\N)$. This is done at no loss of
generality nor precision since there exists a bijective isometry
between $\I_{\beta,\, 2}$ and $l^2(\N)$.

\subsection{Gaussian structure on $l^2(\N)$}
\label{sec:gauss-struct-l2n-1}

Actually, the canonical isometry is given by the Fourier expansion of
the $\beta$-th \textit{derivative} of an element of $\I_{\beta,\,
  2}$. As is, that would not be explicit enough for the computations
to come to be tractable. We take benefit from the dual aspect of a
time indexed point process. On the one hand, as mentioned above, the
sample-path of a point process is of the form
\begin{equation*}
  t\mapsto  \sum_{n\ge 1} \car_{[t_n,\, 1]}(t)
\end{equation*}
where $(t_n,\, n\ge 1)$ is a strictly increasing sequence of reals,
all but a finite number greater than $1$, and thus belongs to
$\I_{\beta,\, 2}$ for any $\beta<1/2$ as shown above. On the other
hand, it can be seen as a locally finite point measure defined by
\begin{equation*}
  f\in \L^2([0,1]) \mapsto \sum_{n\ge 1} f(t_n).
\end{equation*}
Said otherwise, we have the following identities.
  For $(h,\omega) \in \I_{1-\beta,2}^-\times \I_{\beta,2}^+$
  \begin{multline}
 \label{eq:1}   \int_0^1 h(s) d\omega_s:= \sum_{n\ge 1} h(t_n)\car_{[0,\, 1]}(t_n)\\
= <h, I_{0^+}^{-1}(\omega)>_{\I_{1-\beta
        ,2}^-,\
      I^+_{\beta-1,2}}=(I_{1^-}^{\beta-1}(h),I_{0^+}^{-\beta}(\omega))_{{\mathcal
        L}^2}.
  \end{multline}
Recall that $(e_n, \, n\in \N)$ is a complete orthonormal basis of
$\L^2$ and set $h_n^{1-\beta}=I^{1-\beta}_{1^-}(e_n).$ Then $(h_n^{1-\beta},\,
n\in \N)$ is a complete orthonormal basis of
$\I_{1-\beta,2}^-$. Consider the map $\embed$ defined by:
\begin{align*}
  \embed \, :\, \I_{\beta,2} ^+& \longrightarrow l^2(\N)\\
  \omega & \longmapsto \sum_{n\in \N} \int_0^1 h_n^{1-\beta}(s) \d \omega(s)\ x_n,
\end{align*}
where $(x_n,\, n\in \N)$ is the canonical orthonormal basis of $\l2$.
\begin{theorem}
  \label{thm-isometry}
  The map $\embed$ is a bijective isometry from $\I_{\beta,2}$ into
  $l^2(\N)$.  Its inverse is given by:
  \begin{align*}
    \embed^{-1} \, :\, l^2(\N) & \longrightarrow \I_{\beta,\, 2}^+ \notag\\
   \sum_{n\in \N}\alpha_n\, x_n & \longmapsto \sum_{n\in \N} \alpha_n
    I^\beta_{0^+}(e_n).\label{eq:2}
  \end{align*}
\end{theorem}
\begin{proof}
In view of \ref{eq:1}, we have
\begin{align*}
  \|\embed \omega\|_{\l2}^2&= \sum_{n\in \N}\left(\int_0^1 h_n^{1-\beta}(s) \d \omega(s)\right)^2\\
&=\sum_{n\in \N}(e_n,I_{0^+}^{-\beta}\omega)_{\L^2}^2\\
&=\|I_{0^+}^{-\beta}\omega\|_{\L^2}^2,
\end{align*}
according to Parseval equality. Thus, by the definition of the norm on $\I_{\beta,2}$, 
$\embed$ is an isometry. Since 
\begin{equation*}
  \int_0^1 h_n^{1-\beta}(s)\d\omega(s)=(e_n,\, I^{-\beta}_{0^+}(\omega))_{\L^2},
\end{equation*}
the inverse of $\embed$ is clearly
given by
\begin{align*}
  \embed^{-1} \, :\, l^2(\N) & \longrightarrow \I_{\beta,2}^+ \notag\\
  \sum_{n\ge 1}\alpha_n\, x_n & \longmapsto \sum_{n\ge 0} \alpha_n
  I^{\beta}_{0^+}(e_n).
\end{align*}
The proof is thus complete.
\end{proof}
We thus have the commutative diagram.
\begin{equation*}
  \begin{CD}
    \I_{\beta,\, 2} @>\embed>> l^2(\N)\\
    @V{V_\beta}VV @VV{S_\beta:=\embed\circ V_\beta\circ \embed^{-1}}V\\
    \I_{\beta,\, 2} @>\embed>> l^2(\N)
  \end{CD}
\end{equation*}
According to the properties of Gaussian measure (see
\cite{MR0461643}), we have the following result.
\begin{theorem}
  \label{thm_gaussapp10:6} Let $\mu_\beta$ denote the Wiener measure
  on $\I_{\beta,\, 2}$.  Denote $m_\beta=\embed^*\mu_\beta$, then
  $m_\beta$ is the Gaussian measure on $l^2(\N)$ such that for any
  $v\in l^2(\N)$,
  \begin{equation*}
    \int_{l^2(\N)} \exp(i\,  v.u) \d m_\beta(u)=\exp(-\frac 12
    S_\beta v.v)
  \end{equation*}
  with the following notations.
  \begin{equation*}
    \| x\|^2_{l^2(\N)}= \sum_{n=1}^\infty |x_n|^2 \text{ and } x.y=\sum_{n=1}^\infty x_ny_n, \text{ for all } x,\, y \in l^2(\N).  
  \end{equation*}
\end{theorem}
For the sake of simplicity, we also denote by a dot the scalar product
in $l^2(\N)^{\otimes k}$ for any integer $k$. 

In view of \thmref{thm-isometry}, it is straightforward that the map
$S_\beta$ admits the representation:
\begin{equation*}
  S_\beta=\sum_{n\ge 1}\sum_{k\ge 1} \langle h_n^{1-\beta},\,
  h_k^{1-\beta}\rangle_{\L^2} \ x_n\otimes x_k.
\end{equation*}
By $\CC^k_b(l^2(\N);\, X)$, we denote the space of $k$-times Fréchet
differentiable functions from $l^2(\N)$ into an Hilbert space $X$ with
bounded derivatives: A function $F$ belongs to $\CC^k_b(l^2(\N);\, X)$
whenever
\begin{equation*}
  \|F\|_{\CC^k_b(l^2(\N);\, X)}:= \sup_{j=1,\, \cdots,\, k} \sup_{x\in l^2(\N)} \| \nabla^{(j)} F(x)\|_{X\otimes l^2(\N)^{\otimes j}}<\infty.
\end{equation*}
\begin{defn}
  The Ornstein-Uhlenbeck semi-group on $(l^2(\N),\, m_\beta)$ is
  defined for any $F\in L^2(l^2(\N),\, m_\beta; \, X)$ by
  \begin{align*}
    P_t^\beta F(u)&:=\int_{l^2(\N)} F(e^{-t}u+\sqrt{1-e^{-2t}}\, v)\d
    m_\beta(v),
  \end{align*}
  where the integral is a Bochner integral.  
\end{defn}
The following properties are well known.
\begin{lemma}
  \label{thm_gaussapp10:5}
The semi-group $P^\beta$  is ergodic in the sense that for any $u\in l^2(\N)$,
  \begin{equation*}
    P_t^\beta F(u)\xrightarrow{t\to \infty} \int f\d m_\beta.
  \end{equation*}
  Moreover, if $F$ belongs to $\CC_b^k(l^2(\N);\, X)$, then, $\nabla^{(k)}
    (P_t^\beta F)=\exp(-kt)P_t^\beta (\nabla^{(k)}F)$ so that we have
  \begin{equation*}
   \int_{l^2(\N)} \int_0^\infty  \sup_{u\in l^2(\N)} \|\nabla^{(k)}
    (P_t^\beta F)(u)\|_{l^2(\N)^{\otimes (k)}\otimes X} \d t\ \d m_\beta(u)
    \le \frac 1k \|F\|_{\CC^k_b(l^2(\N);\, X)}.
  \end{equation*}
\end{lemma}
For Hilbert valued functions, we define $A^\beta$ as follows.
\begin{defn}
  \label{lemma:generateur}
  Let $A^\beta$ denote the linear operator defined for $F\in \CC^2_b(l^2(\N);\, X)$ by:
  \begin{equation*}
    (A^\beta F) (u)=u.(\nabla F)(u)-\trace_{l^2(\N)}(S_\beta \circ \nabla^{(2)} F (u)),\ \text{for all } u \in l^2(\N).
  \end{equation*}
We still denote by $A^\beta$ the unique extension of $A^\beta$ to its
maximal domain.
\end{defn}
\begin{theorem}
  \label{thm_gaussapp_first:2}
The map $A^\beta$ is the infinitesimal generator of $P^\beta$ in the
sense that for $F\in \CC^2_b(l^2(\N);\, X)$: for any $u\in l^2(\N)$,
\begin{equation}\label{eq_gaussapp20:6}
  P^\beta_t F(u)=F(u)-\int_0^t A^\beta P^\beta_sF(u)\d s.
\end{equation}
\end{theorem}
\begin{proof}
  By its very definition,
  \begin{equation*}
    A^\beta F(u) =\left.\frac{d}{dt}P_t^\beta F(u)\right|_{t=0}.
  \end{equation*}
If $F\in \CC^2_b(l^2(\N);\, X)$, it is clear that
\begin{multline}\label{eq:4}
  \frac{d}{dt}P_tF(u)=-e^{-t}\int_{l^2(\N)} u.\nabla F(e^{-t}u+\sqrt{1-e^{-2t}}\, v)\d
    m_\beta(v)\\
+\frac{e^{-2t}}{\sqrt{1-e^{-2t}}}\int_{l^2(\N)} v.\nabla F(e^{-t}u+\sqrt{1-e^{-2t}}\, v)\d
    m_\beta(v).
\end{multline}
The rest of the proof boils down to show that 
\begin{equation}\label{eq_gaussapp_first:6}
  \int_{l^2(\N)} v.\nabla F(e^{-t}u+\sqrt{1-e^{-2t}}\, v)\d
    m_\beta(v)=  \int_{l^2(\N)} \trace_{l^2(\N)}(S_\beta \circ \nabla^{(2)} F (u))\d
    m_\beta(v).
\end{equation}
Taking that for granted, the result follows by setting $t=0$ in \eqref{eq:4}.
Now, for $\nu_\Gamma$ the centered Gaussian measure on $\R^n$ of covariance matrix $\Gamma$, it is tedious but straightforward to show that 
\begin{equation}\label{eq_gaussapp_first:5}
  \int_{\R^n} \nabla F(y).y \d\nu_\Gamma(y)=\int_{\R^n} \trace(\Gamma\circ \nabla^{(2)}F(y))\d\nu_\Gamma(y).
\end{equation}
Let $(g_n^\beta,\, n\in \N)$ be CONB of $\I_{\beta,\, 2}$ which reduces $S_\beta$, i.e.
\begin{equation*}
  S_\beta=\sum_{n\in \N} \lambda_n(S_\beta)\, g_n^\beta\otimes g_n^\beta,
\end{equation*}
where $(\lambda_n(S_\beta),\, n\in \N)$ is the set of eigenvalues of $S_\beta$. Let $\pi_N$ the orthogonal projection in $\I_{\beta,2}$, on $\text{span}\{g_n^\beta,\, n\le N\}$, $u_N=\pi_N u$ and $u_n^\perp=u-u_N$. Denote by $\nu_n=\pi_N^*m_\beta$ and $\mu_n^\perp=(\Id-\pi_N)^*m_\beta$. By the properties of Gaussian measures, 
\begin{multline*}
  \int_{l^2(\N)} v.\nabla F(v)\d
    m_\beta(v)=\int_{l^2(\N)} (v_N+v_N^\perp).\nabla F(v_N+v_N^\perp) \d \nu_n(v_N)\d\nu_N^\perp(v_N^\perp)\\
= A_1^N+A_2^N.
\end{multline*}
Since  $F\in \CC^2_b(l^2(\N);\, X)$,
\begin{equation*}
  |A_2^N|\le \|\nabla F\|_\infty \left(\int_{l^2(\N)}|v_N^\perp|^2 \d\nu_N^\perp(v_N^\perp)\right)^{1/2}.
\end{equation*}
Since $\nu_N^\perp$ is a Gaussian measure on $\I_{\beta,\, 2}$ whose covariance kernel is $\pi_N^\perp S_\beta \pi_N^\perp$, we have
\begin{equation*}
  \int_{l^2(\N)}|v_N^\perp|^2 \d\nu_N^\perp(v_N^\perp)=\trace(\pi_N^\perp S_\beta\pi_N^\perp).
\end{equation*}
Since $\pi_N^\perp$ tends to the null operator as $N$ goes to infinity, $A_2^N$ tends to $0$. Moreover, $\Gamma_N=\pi_N S_\beta \pi_N$ tends in trace norm to $S_\beta$, hence for any $u\in \l2$,
\begin{equation*}
  \trace(\tilde\Gamma_N\circ \nabla^{(2)}F(u))\xrightarrow{N\to \infty} \trace(S_\beta\circ \nabla^{(2)}F(u)),
\end{equation*}
where $\tilde \Gamma_N(u_N+u_N^\perp)=\Gamma_N(u_N)$ for any $u=u_N+u_N^\perp$ in $\l2$.
According to \eqref{eq_gaussapp_first:5},
\begin{equation*}
   \int_{\R^N} \nabla F(u_N+u_N^\perp).u_N \d\nu_N(u_N)=\int_{\R^N} \trace(\tilde\Gamma_N\circ \nabla^{(2)}F(u_N+u_N^\perp))\d\nu_N(u_N).
\end{equation*}
Hence,
\begin{equation*}
  A_1^N=\int_{l^2(\N)} \trace(\tilde\Gamma_N\circ \nabla^{(2)}F(u))\d\nu(u),
\end{equation*}
 and by dominated convergence, we get \eqref{eq_gaussapp_first:6}.
\end{proof}
\subsection{Notations}
\label{sec:notations}

Before going further, we summarize the notations.
\begin{center}
  \begin{itemize}
  \item $x.y$ : canonical scalar product on $\l2^{\otimes (k)}$
\item $\langle f,g\rangle_{\I_{\alpha,2}}$ : canonical scalar product on $\I_{\alpha,2}$
\item $\nabla F$ : gradient of a Fr\'echet differentiable $F$ defined on $\l2$
\item $\mu_\beta$ (respectively $m_\beta$) : Gaussian measure on $\I_{\beta,2}$ (resp. $\l2$)
\item $h_n^\alpha=I^\alpha_{1^-}(e_n)$ where $(e_n,\, n\in \N)$ is a CONB of $\L^2$
\item $(x_n,\, n\in \N)$ the canonical basis of $\l2$
\item $c_\alpha$ : Hilbert-Schmidt norm of $I^\alpha_{1^-}$
  \end{itemize}
\end{center}

\section{Stein method}
\label{sec:stein-method}
For $\mu$ and $\nu$ two probability measures on $\R^\N$ equipped with
its Borel $\sigma$-field, we define
a distance by
\begin{equation*}
  \rho_\TT(\nu,\, \mu)=\sup_{\|F\|_\TT\le 1} \int F \d \nu -\int F \d \mu.
\end{equation*}
where $\TT$ is a normed space of test functions (the norm of which is
denoted by $\|.\|_{\TT}$). If $\TT$ is the set $1$-Lipschitz functions
on $\l2$, then $\rho_\TT$ corresponds to the optimal
transportation problem for the cost function $c(x,\,
y)=\|x-y\|_{\l2},$ $x,\, y\in \l2$ (see
\cite{Villani:2007fk}). For technical reasons (as in \cite{MR2727319})
mainly due to the infinite dimension, we must restrict the space $\TT$
to smaller subsets.  We thus introduce the distances $\rho_j$ for
$j\ge 1$ as
\begin{equation*}
  \rho_j(\nu,\, \mu)=\sup_{\|F\|_{\CC^j_b(\l2;\, \R)}\le 1} \int F \d \nu -\int F \d \mu.
\end{equation*}
However, these weaker distances still metrize the space of weak
convergence of probability measures on $\l2$.
\begin{theorem}
  Let $(\nu_n,\, n\ge 1)$ be a sequence of probability measures on
  $\l2$ such that some $j\ge 1$,
  \begin{equation*}
    \rho_j(\nu_n,\, \mu)\xrightarrow{n\to \infty}0.
  \end{equation*}
  Then, $(\nu_n,\, n\ge 1)$ converges weakly to $\mu$ in
  $\l2$:
  \begin{equation*}
    \int F\d\nu_n \xrightarrow{n\to \infty} \int F\d \mu,
  \end{equation*}
  for any $F$ bounded and continuous from $\l2$ into $\R$.
\end{theorem}
\begin{proof}
  As Hilbert spaces admit arbitrarily smooth partition of unity
  \cite{MR0431240}, for $j\ge 2$,  one can mimic the proof of
  \cite[page 396]{MR982264} (see also \cite{MR2608474}) which corresponds to $\rho_1$.
\end{proof}
Say that $\mu=m_\beta$ is our reference measure, that is the measure
we want the other measures to be compared to. Stein method relies on
the characterization of $m_\beta$ as the stationary measure of the 
ergodic semi-group $P^\beta$. In view of \eqref{eq_gaussapp20:6}, for
$j\ge 2$,
\begin{equation*}
  \rho_j(\nu,\, m_\beta)=\sup_{\|F\|_{\CC^j_b(\l2;\, \R)}\le
    1}\int_{\l2}\int_0^\infty A^\beta P_t^\beta F(x) \d t\d \nu(x).
\end{equation*}
Thanks to the integration by parts induced by Malliavin calculus, we
can control the right-hand-side integrand and obtain bounds on
$\rho_j(\nu,\, m_\beta)$. To be more illustrative, the Stein method
works as follows: construct a process $(t\mapsto \mathfrak X(x,\, t))$
constant in distribution if its initial condition $x$ is distributed
according to $m_\beta$. Moreover, for any initial distribution, the
law of $\mathfrak X(x,t)$ tends to $m_\beta$ as $t$ goes to
infinity. Stein method then consists in going back in time, from
infinity to $0$, controlling along the way the derivative of the
changes, yielding a bound on the distance between the two initial
measures. Other versions (coupling, size-bias, etc) are just other
ways to construct another process $\mathfrak X$. In these approaches,
for every $\nu$, the couplings are ad-hoc whereas Malliavin calculus
gives a certain kind of universality as it depends only on the
underlying alea. Malliavin structures are well established for
sequences of Bernoulli random variables, Poisson processes, Gaussian
processes and several other spaces (see \cite{MR2531026}). In what
follows, we show an example of the machinery for each of these three
examples.

The core of the method can be summarized in the following theorem.
\begin{hyp}
  For $X$ a Hilbert space, $H\in \l2\otimes X$ and $\alpha$ a
  non-negative real, we say that the probability measure $\nu$ satisfies  $\Hyp(X,\, H,\, \alpha)$ 
  whenever for any $G\in \CC^2_b(\l2;\, \l2)$
  \begin{multline}
    \label{eq_gaussapp10:1}
    \Bigl| \int_{\l2}x.G(x)\d\nu(x)-\int_{\l2}\trace(\trace_X(H\otimes
      H)\circ \nabla G(x))\d\nu(x)\Bigr| \\ \le \alpha \,
    \|\nabla^{(2)}G\|_\infty.
  \end{multline}
\end{hyp}
\begin{theorem}[Stein method]
  \label{thm_gaussapp10:1}
  Assume that $\Hyp(X,\, H,\, \alpha)$ holds. Then, if $\alpha >0$,
  \begin{equation}
    \label{eq_gaussapp10:2}
    \rho_3(\nu,\, m_\beta)\le   \frac12\ \| \trace_{X}(H\otimes H)-S_\beta\|_{\mathcal S_1}  + \frac{\alpha}{3}\cdotp
  \end{equation}
If $\alpha=0$,
\begin{equation}
  \label{eq_gaussapp20:7}
   \rho_2(\nu,\, m_\beta)\le \frac12 \| \trace_{X}(H\otimes H)-S_\beta\|_{\mathcal S_1} .
\end{equation}
\end{theorem}
\begin{rem}
  The two terms in the right-hand-side of \eqref{eq_gaussapp10:2} are
  of totally different nature. The trace term really measures the
  effect of the approximation scheme whereas the second term comes
  from a sort of curvature of the space on which is built the
  approximate process. As will become evident in the examples below,
  this term is zero when the Malliavin gradient satisfies the chain
  rule formula and non-zero otherwise.
\end{rem}
\begin{proof}
  For $\alpha >0$, for $F\in \CC^3_b$, according to \lemref{thm_gaussapp10:5} and \thmref{thm_gaussapp_first:2}, we have
\begin{multline}\label{eq:5}
  \esp{\nu}{F }-\esp{m_\beta}{F}\\
  =-\int_{\l2} {\int_0^\infty  x. \nabla P_t^\beta F(x)
    -\trace\left(S_\beta\circ  \nabla^{(2)}P^ \beta_t F(x)\right)\d t}
  \d\nu(x).
\end{multline}
Applying $\Hyp(X,\, H, \alpha)$ to $G=\nabla P_t^\beta F$, we have
\begin{multline*}
  \left| \esp{\nu}{F }-\esp{m_\beta}{F}\right|
  \le   \left|\esp{\nu}{\int_0^\infty \trace(\trace_X(H\otimes H)-S_\beta)\circ \nabla^{(2)}P_t^\beta F )\d t}\right|\\
  \shoveright{ +\alpha \, \esp{\nu}{\int_0^\infty \|\nabla^{(3)}
      P_t^\beta F\|_\infty\d t}}\\
  \le \frac12 \|\nabla^{(2)} F\|_\infty \|\trace_X(H\otimes
  H)-S_\beta\|_{\mathcal S_1}+ \frac{\alpha}{3} \|\nabla^{(3)} F\|_\infty ,
\end{multline*}
according to \lemref{thm_gaussapp10:5} and Equation \eqref{eq_gaussapp_first:4}. 
If $\alpha=0$, the very same lines show that the second order
differential of $F$ is sufficient to have a bound of $ \esp{\nu}{F }-\esp{m_\beta}{F}$.
\end{proof}

\section{Normal approximation of Poisson processes}
\label{sec:norm-appr-poiss}
\label{sec:integr-parts-poiss}
Let $\chi_{[0,\, 1]}$ the space of locally finite measures on $[0,\,
1]$ equipped with the vague topology. We identify a point measure
$\omega=\sum_{n\in \N} \delta_{t_n}$ with the one dimensional process
\begin{equation*}
  N\, : \, t\in [0,\, 1] \longmapsto \int_0^t \d\omega(s)=\sum_{n\in
    \N} \car_{[0,\, t]}(t_n).
\end{equation*}
The measure $\nu_\lambda$ is the only measure on $(\chi_{[0,\, 1]},\,
{\mathfrak B}(\chi_{[0,\, 1]}))$ such that the canonical process $N$
is a Poisson process of intensity $\lambda \d \tau$. 
It is well known that for a Poisson process $N$ of intensity
$\lambda$, the process
\begin{equation*}
  N_\lambda(t)=\frac 1{\sqrt{\lambda}}\left( N(t) - \lambda t \right)
\end{equation*}
converges in distribution on $\mathfrak D$ to a Brownian motion as
$\lambda$ goes to infinity. For any $\beta<1/2$, we want to precise
the rate of convergence.

\subsection{Malliavin calculus for Poisson process}
\label{sec:hilb-valu-mall}
For a real valued functional $F$ on $\chi_{[0,\, 1]}$, it is customary
to define the discrete gradient as
\begin{equation*}
  D_\tau F(N)=F(N+\epsilon_\tau)- F(N), \text{ for any } \tau \in
  [0,\, 1],
\end{equation*}
where $N+\epsilon_\tau$ is the point process $N$ with an extra atom at
time $\tau$. 
We denote by $\D_{2,1}^\lambda$ the set of square integrable
functionals $F$ such that 
\begin{equation*}
\|F\|_{2,1,\lambda}^2:= \esp{\nu_\lambda}{F^2} +\esp{\nu_\lambda}{\int_0^1 |D_\tau
  F|^2\lambda \d\tau} 
\end{equation*}
is
finite.  A process $G\in L^2(\nu_\lambda\times d\tau)$ is said to
belong to $\Dom \delta^\lambda$ whenever there exists $c>0$ such that 
\begin{equation*}
  \esp{\nu_\lambda}{\int_0^1 D_\tau F\ G_\tau\, \lambda \d\tau}\le c\, \|F\|_{L^2(\nu_\lambda)},
\end{equation*}
for any $F\in \D_{2,1}^\lambda$. The adjoint of $D$, denoted by $\delta^\lambda$ is
then defined  by the following relationship:
\begin{equation}
  \label{eq_gaussapp6:4} \esp{\nu_\lambda}{F\ \delta^\lambda(G)}=\lambda\ \esp{\nu_\lambda}{\int_0^1
    D_\tau F\  G_\tau\d \tau}.
\end{equation}
Moreover, it is well known that for $G$ deterministic,
$\delta^\lambda$ coincides with the compensated integral with respect to
the Poisson process, i.e.
\begin{equation*}
  \delta^\lambda G=\int_0^1 G_\tau (\d\omega(\tau)-\lambda \d \tau),
\end{equation*}
 and that 
\begin{math}
  D\delta^\lambda G=G.
\end{math}
\subsection{Convergence theorem}
\label{sec:convergence-theorem}
\begin{theorem}\label{thm_gaussapp10:2}
Let $H_\lambda=\lambda^{-1/2}\sum_{n\ge 1}
   h_n^{1-\beta}\otimes  x_n=\lambda^{-1/2}H_1$.  We denote by $\nu_\lambda^*$ the distribution of $\embed N_\lambda$
  in $\l2$. The measure $\nu_\lambda^*$ satisfies
  $\Hyp(\L^2([0,1]),\, H_1,\, a)$ with 
  \begin{equation}
    \label{eq_gaussapp_first:3}
  a=\frac{(1-\beta)^{3/2}}{5-6\beta}
\frac{c_{1-\beta}^3}{\sqrt{\lambda}}\le \frac{c_{1-\beta}^3}{2\sqrt{\lambda}}\cdotp  
  \end{equation}
  Hence,
  \begin{equation*}%\label{eq_gaussapp10:9}
    \rho_3(\nu_\lambda^*,\, m_\beta)\le \frac{  a}{3\sqrt{\lambda}}\cdotp
  \end{equation*}
\end{theorem}
\begin{rem}  From its very definition, it is clear that 
  \begin{equation*}
    \embed N_\lambda =\frac{1}{\sqrt{\lambda}}\sum_{n\ge 1} \delta^\lambda( h_n^{1-\beta})\ x_n
  \end{equation*}
where $(x_n,\, n\ge 1)$ is the canonical orthonormal basis of $l^2(\N)$.
Note also that 
\begin{equation*}
  D \embed N_\lambda =\frac{1}{\sqrt{\lambda}}\sum_{n\ge
    1}h_n^{1-\beta}\otimes x_n.
\end{equation*}
It is because of this particular form of $D\embed N_\lambda$ as an infinite series
of simple bilinear forms on $\I^-_{1-\beta,2}\otimes \l2$  that the
computations to come are feasible. To compare, if we view $N_\lambda$
as an element of $\I_{\beta,\, 2}$, then 
\begin{equation*}
  D_\tau N_\lambda(t)=I^\beta_{0^+}((\tau-.)_+^{-\beta})(t).
\end{equation*}
Since there is no \textsl{decoupling} in this expression between the
$\tau$ variable and the $t$ variable, the 
computations are intractable; hence the need to resort to the Gaussian
structure on $\l2$.
\end{rem}

\begin{proof}[Proof of \thmref{thm_gaussapp10:2}]

Let $F\in \CC^2_b(\l2;\R)$ and $x\in \l2$. Denoting by $G(y)=F(y)x$
for $y\in \l2$, we have
\begin{align*}
  \esp{}{\embed N_\lambda. G(\embed
    N_\lambda)}&=\frac{1}{\sqrt{\lambda}}\sum_{n\ge 1} 
\esp{}{\delta^\lambda( h_n^{1-\beta})F(\embed N_\lambda)}\ x_n.x\\
&=\frac{1}{\sqrt{\lambda}}\sum_{n\ge 1} 
\esp{}{\int_0^1 h_n^{1-\beta}(\tau) D_\tau F(\embed N_\lambda)\lambda
  \d \tau}\ x_n.x\\
&=\sqrt{\lambda}\ \esp{}{\int_0^1  D_\tau F(\embed N_\lambda). H_1(\tau)\d\tau}.
\end{align*}
According to the Taylor formula,
\begin{multline*}%\label{eq_gaussapp20:8}
   D_\tau F(\embed N_\lambda) = F(\embed N_\lambda+H_\lambda(\tau))-F(\embed N_\lambda)\\
  =\frac{1}{\sqrt{\lambda}}\nabla F(\embed N_\lambda).H_1(\tau) +\frac{1}{\lambda}\int_0^1 (1-r)\ 
    \nabla^{(2)} F(\embed N_\lambda + r H_\lambda(\tau)). H_1(\tau)^{\otimes (2)}
   \d r.
\end{multline*}
Thus,  we get
\begin{multline}\label{eq_gaussapp_first:2}
 \esp{}{\embed N_\lambda.G(\embed
    N_\lambda)}= \esp{}{\int_0^1 \nabla G(\embed
    N_\lambda). H_1(\tau)^{\otimes (2)}\d \tau}\\
+\lambda^{-1/2}\esp{}{\int_0^1\int_0^1 (1-r)\ \nabla^{(2)}G(\embed
  N_\lambda). H_1(\tau)^{\otimes (3)}\d\tau \d r}.
\end{multline}
By linearity and density, \eqref{eq_gaussapp_first:2} holds for any
$G\in \C^2_b(\l2;\l2).$
Note that for any $A=\sum_{n,\, jk\in \N}a_{n,\, k} \, x_n\otimes x_k\in
\l2\otimes \l2$
\begin{equation*}
 \int_0^1 A. H_1(\tau)^{\otimes(2)}\d\tau=\sum_{n,\, k=1}^\infty a_{i,j}\Bigl( h_n^{1-\beta},\,
    h_k^{1-\beta}\Bigr)_{\L^2}=\trace_{\l2}(S_\beta\circ A).
\end{equation*}
Hence, 
\begin{equation*}
  \int_0^1 \nabla G(\embed
    N_\lambda). H_1(\tau)^{\otimes (2)}\d
    \tau=\trace_{\l2}(S_\beta\circ \nabla G(\embed
    N_\lambda)).
\end{equation*}
Since $\nabla^2 G$ is bounded, we have
\begin{multline*}
  \left| \esp{}{\int_0^1\int_0^1 (1-r)\ 
        \nabla^{(2)}G(\embed N_\lambda + r H_\lambda),
        H_1(\tau)^{\otimes 3}\d r\d \tau}\right|\\
  \le \frac 1{2} \|\nabla^{(2)}
  G\|_\infty\int_0^1 \|H_1(\tau)\|^3_{\l2}\d \tau.
\end{multline*}
Moreover, according to \eqref{eq_gaussapp_first:1}, 
\begin{equation*}
  \int_0^1 \|H_1(\tau)\|^3_{\l2}\d \tau= \int_0^1 \left(\sum_{n\ge
    1}h_n^{1-\beta}(\tau)^2\right)^{3/2}\ d\tau=\frac{(1-\beta)^{3/2}}{5/2-3\beta}c_{1-\beta}^3
\end{equation*}
 Hence, it follows that 
\begin{equation*}
  \left|  \esp{}{\embed N_\lambda.G(\embed
    N_\lambda)}-\esp{}{\trace(S_\beta\circ \nabla G(\embed
    N_\lambda))}\right|\le \frac{(1-\beta)^{3/2}}{5/2-3\beta} \frac{c_{1-\beta}^3}{2\sqrt{\lambda}}\|\nabla^{(2)} G\|_\infty,
\end{equation*}
which is Equation \eqref{eq_gaussapp10:1} with $X=\L^2$ and
$\alpha=a$ given by \eqref{eq_gaussapp_first:3}.
\end{proof}
\begin{rem}
  It is remarkable that by homogeneity, the partial trace of
  $H_1\otimes H_1$ is equal to $S_\beta$. The only remaining term in
  \thmref{thm_gaussapp10:1} comes from the fact that the discrete
  gradient does not satisfy the chain rule.

One could also remark that the choice of the space in which we embed
the Poisson and Brownian sample-paths (i.e. the choice of the value of
$\beta$) modifies only the constant but not the order of convergence,
which remains proportional to $\lambda^{-1/2}$.
\end{rem}
\section{Linear interpolation of the Brownian motion}
\label{sec:linear-interpolation}

For $m\ge 1$, the linear interpolation $B_m^\dag$ of a Brownian motion $B^\dag$
is defined by
\begin{equation*}
  B_m^\dag(0)=0 \text{ and }  \d B_m^\dag(t)=m\sum_{j=0}^{m-1} (B^\dag(j+1/m)-B^\dag(j/m)) \car_{[j/m,\, (j+1)/m)}(t)\d t.
\end{equation*}
Thus, $\embed B_m^\dag$ is given by
\begin{equation*}
  \embed B_m^\dag= m\sum_{j=0}^{m-1} (B^\dag(j+1/m)-B^\dag(j/m))
  \sum_{n\in \N}\int_{j/m}^{(j+1)/m}h_n^{1-\beta}(t)\d t\ x_n.
\end{equation*}
Consider the $\L^2$-orthonormal functions
\begin{equation*}
  e^m_j(s)=\sqrt{m}\, \car_{[j/m,\, (j+1)/m)}(s),\, j=0,\, \cdots,\, m-1,\,  s\in [0,\, 1]
\end{equation*}
and $F_m^\dag=\text{span}(e^m_j,\, j=0,\, \cdots,\, m-1).$ We denote by
$p_{F_m^\dag}$ the orthogonal projection over $F_m^\dag$.  Since $B_m^\dag$ is
constructed as a function of a standard Brownian motion, we work on
the canonical Wiener space $({\mathcal C}^0([0,\, 1];\, \R),\ \I_{1,\,
  2},\, m^\dag)$. The gradient we consider, $D^\dag$, is the
derivative of the usual gradient on the Wiener space and the
integration by parts formula reads as:
\begin{equation}
  \label{eq_gaussapp10:10}
  \esp{m^\dag}{F \int_0^1 u(s)\d B^\dag(s)}=\esp{m^\dag}{\int_0^1
    D^\dag_s F\  u(s)\d s}
\end{equation}
for any $u\in \L^2$.
Let 
\begin{equation*}
  H^\dag_m=\sum_{n\in \N}p_{F_m^\dag}h_n^{1-\beta}\otimes x_n\in
  \L^2\otimes l^2(\N).
\end{equation*}
It means that 
\begin{equation*}
  H^\dag_m(k,\, s)=m\sum_{n\in \N}\sum_{j=0}^{m-1}
  (\int_{j/m}^{(j+1)/m} h_k^{1-\beta}(t) \d t)\ \car_{[j/m,\,
    (j+1)/m)}(s)\ x_n.
\end{equation*}
Since the $e^m_j$'s are orthogonal in $\L^2$, we can compute the partial
trace as follows.
\begin{multline}\label{eq:9}
  \trace_{\L^2} ( H^\dag_m\otimes  H^\dag_m) \\ = m\sum_{n\in \N}\sum_{k\in
    \N}\sum_{j=0}^{m-1} (\int_{j/m}^{(j+1)/m}h^{1-\beta}_k(t)\d t)(  \int_{j/m}^{(j+1)/m}h^{1-\beta}_n(t)\d t) \
  x_n\otimes x_k.
\end{multline}
\begin{theorem}
  \label{thm_gaussapp10:8} Let $\nu_m^\dag$ be the law of $\embed B_m^\dag$
  on $\l2$.  The measure $\nu_m^\dag$ satisfies $\Hyp(\L^2, \,
  H^\dag_m,\, 0)$.
Hence,
    \begin{equation*}
    \rho_2(\nu^\dag_m,\, m_\beta)\le \frac{m^{2\beta-1}}{2(1-2\beta)\, \Gamma(1-\beta)^2}\cdotp
  \end{equation*}
\end{theorem}

\begin{proof}
    For $G$ sufficiently regular, according to the definition of $B^m$
  and to \eqref{eq_gaussapp10:10}, we have
  \begin{multline}
    \label{eq_gaussapp10:11}
      \esp{}{\embed B_m^\dag.G(\embed B_m^\dag)}\\
\begin{aligned}
&=\esp{}{\sum_{n\in
          \N}m\sum_{i=0}^{m-1} (B(i+1/m)-B(i/m))
        \int_{i/m}^{(i+1)/m}h_n^{1-\beta}(t)\d t\ G_n(\embed B_m^\dag)}\\
      &=m\sum_{n\in \N}\sum_{i=0}^{m-1}\int_{i/m}^{(i+1)/m}h_n^{1-\beta}(t)\d t\
      \esp{}{\int_{i/m}^{(i+1)/m}D^\dag_s G_n(\embed B_m^\dag) \d s}\\
&=\int_0^1 H_m^\dag(t)\d t.\esp{m^\dag}{\int_{i/m}^{(i+1)/m}D^\dag_s G(\embed B_m^\dag) \d s}.
    \end{aligned}
  \end{multline}
  Since $D^\dag$ obeys the chain rule formula,
  \begin{equation}\label{eq:6}
    \begin{split}
    D^\dag_s G_n(\embed B_m^\dag)&=\sum_{k\in \N}\nabla_k G_n (\embed B_m^\dag)D^\dag_s
    (\embed B_m^\dag)\\
&    =\sum_{k\in \N}\nabla_k G_n (\embed B_m^\dag) (m\sum_{l=0}^{m-1} \car_{[l/m,\,
      (l+1)/m)}(s)\int_{l/m}^{(l+1)/m} h_k(s)\d s)\\
&=\nabla G_n (\embed B_m^\dag).H_m^\dag(s).   
    \end{split}
  \end{equation}
  Combining \eqref{eq_gaussapp10:11} and \eqref{eq:6}, we get
  \begin{multline*}
    \esp{}{\embed B_m^\dag.G(\embed B_m^\dag)}\\
    \begin{aligned}
&= m\esp{}{\sum_{k\in
        \N}\sum_{n\in \N}\sum_{i=0}^{m-1} \nabla_k G_n(\embed B_m^\dag)
      \int_{i/m}^{(i+1)/m}h_n^{1-\beta}(t)\d t\int_{i/m}^{(i+1)/m}h_k(s)\d s}\\
 &   = \esp{}{\trace(\trace_{\L^2}(H_m^\dag\otimes
      H_m^\dag)\circ \nabla G(\embed B_m^\dag))}.      
    \end{aligned}
  \end{multline*}
  It follows that $\nu_m^\dag$ satisfies $\Hyp(\L^2,\, H_m^\dag, \,
  0)$. To conclude, it remains to estimate
  \begin{math}
    \|\trace_{\L^2}(H_m^\dag\otimes H_m^\dag)-S_\beta\|_{\mathcal S_1} .
  \end{math}
According to Pythagorean Theorem, we have
\begin{multline*}
S_\beta-  \trace_{\L^2}(H_m^\dag\otimes H_m^\dag)\\
  \begin{aligned}
& = \sum_{n\in
    \N}\sum_{k\in \N} \left((p_{F_m^\dag}h_n^{1-\beta},\,
  p_{F_m^\dag}h_k^{1-\beta})_{\L^2}-(h_n^{1-\beta},\,
  h_k^{1-\beta})_{\L^2}\right)x_n\otimes x_k\\
&=\sum_{n\in \N}\sum_{k\in \N} ((\Id-p_{F_m^\dag})h_n^{1-\beta},\,
  (\Id-p_{F_m^\dag})h_k^{1-\beta})_{\L^2}\ x_n\otimes x_k.   
  \end{aligned}
\end{multline*}
Hence, $S_\beta-\trace_{\L^2}(H_m^\dag\otimes H_m^\dag)$ is a
symmetric non-negative operator, thus
  \begin{multline*}
  \|\trace_{\L^2}(H_m^\dag\otimes
      H_m^\dag)-S_\beta\|_{\mathcal S_1}\\
    \begin{aligned}
      &\le \sum_{n\in \N} \|(\Id-p_{F_m^\dag})h^{1-\beta}_n\|_{\L^2}^2\\
      &=\sum_{n\in \N}\int_0^1
      \left(h^{1-\beta}_n(s)-\sum_{j=0}^{m-1} \int_{j/m}^{(j+1)/m}
        h^{1-\beta}_n(t)\d t \ e_j^m(s)\right)^2\d s\\
&=\sum_{n\in \N}\sum_{j=0}^{m-1}
\int_{j/m}^{(j+1)/m}\left(h^{1-\beta}_n(s) -m\int_{j/m}^{(j+1)/m}
  h^{1-\beta}_n(t)\d t \right)^2\d s\\
&=m\, \sum_{n\in
  \N}\sum_{j=0}^{m-1}\int_{j/m}^{(j+1)/m}\left(\int_{j/m}^{(j+1)/m}(h^{1-\beta}_n(s)-h^{1-\beta}_n(t))m\d
  t\right)^2\d s\\
&\le m^2 \sum_{n\in
  \N}\sum_{j=0}^{m-1}\int_{j/m}^{(j+1)/m}\int_{j/m}^{(j+1)/m}(h^{1-\beta}_n(s)-h^{1-\beta}_n(t))^2
\d s\d t,
    \end{aligned}
  \end{multline*}
where the last inequality follows from Jensen inequality.
Since $h_n^{1-\beta}=I^{1-\beta}_{1^-}(e_n)$, where $(e_n,\,
 n\in \N)$ is a CONB of $\L^2$, according to Parseval identity,
\begin{align*}
  \sum_{n\in \N}(h^{1-\beta}_n(s)-h^{1-\beta}_n(t))^2 &=
  \frac{1}{\Gamma(1-\beta)^2} \sum_{n\in
    \N}\left((.-s)_+^{-\beta}-(.-t)_+^{-\beta},\,
    e_n\right)^2_{\L^2}\\
&=\frac{1}{\Gamma(1-\beta)^2} \int_0^1
\left((\tau-s)_+^{-\beta}-(\tau-t)_+^{-\beta}\right)^2\d \tau.
\end{align*}
Expanding the square and using the monotonicity of the power function,
we get
\begin{multline*}
  \sum_{n\in \N}(h^{1-\beta}_n(s)-h^{1-\beta}_n(t))^2 \le
  \frac{1}{(1-2\beta)\Gamma(1-\beta)^2}
  (|1-s|^{1-2\beta}-|1-t|^{1-2\beta})\\
\le \frac{(1-2\beta)^{-1}}{\Gamma(1-\beta)^2}|t-s|^{1-2\beta}.
\end{multline*}
It follows that 
\begin{equation*}
   \|\trace_{\L^2}(H_m^\dag\otimes
      H_m^\dag)-S_\beta\|_{\mathcal S_1} \le
    \frac{m^{2\beta- 1}}{(1-2\beta)\Gamma(1-\beta)^2} \cdotp
\end{equation*}
The proof is thus complete.
\end{proof}
\section{Donsker theorem}
\label{sec:donsker-theorem}

The same approach can be applied to have precise asymptotic for the
Donsker theorem. Let $X=(X_n,\, n\in \N)$ be a sequence of independent
and identically distributed Rademacher random variables,
i.e. $\P(X_n=\pm 1)=1/2$ for any $n$. For any $k$ in $\N$, we set
\begin{multline*}
  X_k^+=(X_1,\, \cdots,\, X_{k-1},\, 1, \, X_{k+1}\, \cdots)\\
  \text{ and } X_k^-=(X_1,\, \cdots,\, X_{k-1},\, -1, \, X_{k+1}\,
  \cdots).
\end{multline*}
The discrete gradient on this probability space is given by
\begin{equation*}
  \GD_k F (X)=\frac 12(F(X_k^+)-F(X_k^-)).
\end{equation*}
Then, the integration by parts formula reads as
\begin{equation}\label{eq_gaussapp20:10}
  \esp{}{\sum_{k\in \N}u_k  \GD_k F (X)}=\esp{}{F(X)\sum_{k\in \N}u_k X_k}
\end{equation}
for any $u=(u_k,\, k\in \N)$ which belongs to $\l2$.
The approximating process of the Donsker Theorem is defined by:
\begin{equation*}
  B_m^\sharp(t)=\frac{1}{\sqrt{m}}\left(\sum_{j=1}^{[mt]} X_j
    +(mt-[mt])X_{[mt+1]}\right).
\end{equation*}
Hence,
\begin{equation*}
  \d B_m^\sharp(t)=\frac{1}{\sqrt{m}}\sum_{j=1}^m X_j
  \ \car_{[(j-1)/m,\, j/m)}(t)\d t.
\end{equation*}
Thus, we get
\begin{equation*}
  \embed B_m^\sharp =\sum_{n\in \N}\sum_{j=1}^m \frac{1}{\sqrt{m}} \,
  X_j \, \int_{(j-1)/m}^{j/m} h_n^{1-\beta}(s)\d s\ x_j\otimes x_n.
\end{equation*}
\begin{theorem}\label{thm_gaussapp10:3}
 We denote by $\nu^\sharp_m$ the distribution of $\embed B_m^\sharp$ on
$\l2$. The measure $\nu^\sharp_m$ satisfies $\Hyp(\l2,\, H_m^\sharp,\,
  1)$
   where
  \begin{equation*}
    H_m^\sharp=\sum_{n\in \N}\sum_{j=1}^m \frac{1}{\sqrt{m}} \,
 \int_{(j-1)/m}^{j/m} h_n^{1-\beta}(s)\d s\ x_j\otimes x_n.
  \end{equation*}
\label{thm_gaussapp10:4}
Furthermore, for any $\epsilon >0$, there exists $m_0$ such that for $m\ge m_0$,
  \begin{equation*}
    \rho_3(\nu^\sharp_m,\, m_\beta)\le  (1+\epsilon)\ 
    \frac{m^{2\beta-1}}{2(1-2\beta)\,\Gamma(1-\beta)^2}\cdotp
  \end{equation*}
\end{theorem}
\begin{proof}
  According to the integration by parts formula \eqref{eq_gaussapp20:10}, we have
  \begin{align}
    \esp{}{\embed B_m^\sharp.G(\embed B_m^\sharp)}&= \frac{1}{\sqrt{m}}\,\esp{}{\sum_{n\ge 1} \embed B_m^\sharp(n)\, G_n(\embed B_m^\sharp)}\notag\\
&= \frac{1}{\sqrt{m}}\,\esp{}{\sum_{n\ge 1}\sum_{k=1}^m h_n^{1-\beta}(k/m)X_n\, G_n(\embed B_m^\sharp)}\notag\\
&= \frac{1}{\sqrt{m}}\,\esp{}{\sum_{k=1}^m \sum_{n\ge 1} h_n^{1-\beta}(k/m) \GD_k G_n( \embed B_m^\sharp)}\notag\\
%&= \esp{}{\sum_{k=1}^m \langle \GD_k G(\embed B_m^\sharp),\ H_m^\sharp(k)\rangle_{\l2}}\\
&=\esp{}{\GD G(\embed B_m^\sharp).\ H_m^\sharp},\label{eq:8}
  \end{align}
  According to the Taylor formula,
  \begin{multline*}
    \GD_j G(\embed B_m^\sharp)=\frac 12\left(G(\embed B_m^\sharp + (1-X_j)H_m^\sharp(j))-G(\embed B_m^\sharp -(1+X_j) H_m^\sharp(j))\right)\\
    \begin{aligned}
      &=\langle \nabla G(\embed B_m^\sharp),\, H_m^\sharp(j)\rangle_{\l2}(1-X_j+1+X_j)/2\\
      &+\frac{(1-X_j)^2}{2}\int_0^1 (1-r)\nabla^2G(\embed B_m^\sharp+r(1-X_j) H_m^\sharp(j)). \,  H_m^\sharp(j)^{\otimes (2)}\d r\\
      &+\frac{(1+X_j)^2}{2}\int_0^1 (1-r) \nabla^2G(\embed
        B_m^\sharp+r(1+X_j)H_m^\sharp(j)). \, H_m^\sharp(j)^{\otimes (2)}\d r.
    \end{aligned}
  \end{multline*}
Plugging this latter equation into \eqref{eq:8}, it follows that 
  \begin{multline*}
    \esp{}{\embed B_m^\sharp.G(\embed B_m^\sharp)}=\esp{}{\sum_{j=1}^m \nabla G(\embed B_m^\sharp).\, H_m^\sharp(j)\otimes H_m^\sharp(j)}\\
\shoveleft{     +\sum_{z=\pm 1}{\mathbb E}\left[\sum_{j=1}^m\frac{(1-zX_j)^2}{2}\right.}\\
\left.\times \int_0^1 (1-r)
        \nabla^{(2)}G(\embed B_m^\sharp+ r(1-zX_j)H_m^\sharp(j)). \,
        H_m^\sharp(j)^{\otimes 3}\d
      r\right].
  \end{multline*}
Since $(1-zX_j)$ is either $0$ or $2$ for any $j\ge 1$ and any $z\in \pm 1$, we get
  \begin{multline*}
    \left| \esp{}{\embed B_m^\sharp.G(\embed B_m^\sharp)}-\esp{}{\trace(\trace_{\l2}(H_m^\sharp\otimes H_m^\sharp)\circ \nabla^2G(\embed B_m^\sharp)))}\right|\\ 
\le \|\nabla^2 G\|_\infty\sum_{j=1}^m \|H_m^\sharp(j)\|_{\l2}^3,
  \end{multline*}
  which is \eqref{eq_gaussapp10:1} with $\alpha=\sum_{j=1}^m \|H_m^\sharp(j)\|_{\l2}^3$.
 It turns out that according to~\eqref{eq:9},
\begin{align*}
  \trace_{\l2}(H_m^\sharp \otimes H_m^\sharp)=\sum_{n\in
    \N}\sum_{k\in \N}\sum_{j=1}^m  (h^{1-\beta}_n,\, e_{j-1}^m)_{\L^2}(h^{1-\beta}_n,\, e_{j-1}^m)_{\L^2}\ x_n\otimes x_k\\
=\trace_{\L^2} (H_m^\dag \otimes H_m^\dag).
\end{align*}
Hence we can use the result of \thmref{thm_gaussapp10:8}. It remains
to control the additional term (due to the fact that $D^\sharp$ does
not satisfy the chain rule formula) 
\begin{math}
  \sum_{j=1}^m \|H_m^\sharp(j)\|_{\l2}^3.
\end{math}
By the very definition of $H_m^\sharp$,
\begin{align*}
  \|H_m^\sharp(j)\|_{\l2}^2&=\frac 1m\sum_{n\in \N}(h_n^{1-\beta},\,
  e_j^m)^2_{\L^2}\\
&=\frac 1m\sum_{n\in \N} (e_n,\, I^{1-\beta}_{0^+}(e_j^m))^2_{\L^2}\\
&=\frac 1m\| I^{1-\beta}_{0^+}(e_j^m)\|^2_{\L^2}\\
&=\frac 1{m\Gamma(1-\beta)^2}\int_0^1\left( \int_{(j-1)/m}^{j/m}
  (\tau-s)^{-\beta}\d s\right)^2\d \tau\\
&\le \frac{m^{2\beta-3}}{(1-\beta) \Gamma(1-\beta)^2}\cdotp
\end{align*}
Thus,
\begin{equation*}
  \sum_{j=1}^m \|H_m^\sharp(j)\|_{\l2}^3 \le  \frac{m^{3\beta-7/2}}{(1-\beta) \Gamma(1-\beta)^2}\cdotp
\end{equation*}
The dominating term is thus the term in $m^{2\beta-1}$ and the result follows.
\end{proof}

\section{Transfer principle}
\label{sec:transfer-principle}

For $X$ and $Y$ two Hilbert spaces and $\Theta$ a continuous linear
map from $X$ to $Y$. Let $\mu$ and $\nu$ two probability measures on
$X$ and $\mu_Y$ (respectively $\nu_Y$) their image measure with
respect to $\Theta$. Since $\Theta$ is linear and continuous, for
$F\in F\in{\mathcal C}^k_b(Y,\, \R)$, $F\circ \Theta$ belongs to
${\mathcal C}^k_b(X,\, \R)$, hence, we have
\begin{multline*}
  \sup_{F\in{\mathcal C}^k_b(Y,\, \R)}\int F\d \mu_Y-\int F\d \nu_Y=\sup_{F\in{\mathcal C}^k_b(Y,\, \R)}\int F\circ \Theta\d \mu-\int F\circ \Theta\d \nu\\
  \le\sup_{F\in{\mathcal C}^k_b(X,\, \R)}\int F\d \mu-\int F\d \nu.
\end{multline*}
As an application, we can precise the convergence established in
\cite{decreusefond03.2}. Note that in this paper, the key tool was
also a matter of Hilbert-Schmidt property of some operator.

The fractional Brownian motion of Hurst index $H \in [0,\, 1]$ may be defined (see \cite{decreusefond03.1}) by 
\begin{equation*}
  B^H(t)=\int_0^t K_H(t,s) \d B(s),
\end{equation*}
where
\begin{equation*} K_H(t,r):=\frac{(t-r)^{H- \frac{1}{2}}}{\Gamma(H+\frac{1}{2})}
F(\frac{1}{2}-H,H-\frac{1}{2}, H+\frac{1}{2},1- 
\frac{t}{r})1_{[0,t)}(r).
\end{equation*}
The Gauss hyper-geometric function $F(\alpha,\beta,\gamma,z)$  (see
\cite{nikiforov88}) is the analytic continuation on ${\mathbb
  C}\times {\mathbb C}\times {\mathbb C} \backslash \{-1,-2,\ldots
\}\times \{z\in {\mathbb C}, Arg |1-z| < \pi\}$ of the power series
 \begin{displaymath}
    \sum_{k=0}^{+ \infty} \frac{(\alpha)_k(\beta)_k}{(\gamma)_k 
k!}z^k,
 \end{displaymath}
and 
\begin{displaymath}
  (a)_0=1 \text{ and } (a)_k =
\frac{\Gamma(a+k)}{\Gamma(a)}=a(a+
1)\dots (a+k-1).
\end{displaymath}
Furthermore, according to \cite{samko93}, $K_H$ is a continuous map
from $\L^2$ in to $\I^+_{H+1/2,\, 2}$ hence the map $\Theta_H=K_H\circ I^{-1}_{0^+}$ can be defined continuously from $\I^+_{\beta,\, 2}$ to $\I^+_{H-(1/2-\beta),\, 2}$.  Since $\Theta_HB=B^H$, we have the following result.
\begin{theorem}
For any $H\in [0,\, 1]$, for any $1/2>\epsilon >0$,
\begin{equation*}
  \rho_3\left(\mathfrak J_{H-\epsilon} (\int_0^. K_H(t,\, s)\d N^\lambda (s)),\ \mathfrak J_{H-\epsilon} (B^H)\right)\le \frac{ a}{3\sqrt{\lambda}}\cdotp
\end{equation*}
\end{theorem}

\end{document}